\documentclass{ifacconf}
 \usepackage{amsmath,amssymb}
\usepackage{graphicx}      % include this line if your document contains figures
\usepackage{natbib}        % required for bibliography
\usepackage{xcolor}
\newtheorem{definition}{Definition}%
\newtheorem{remark}{Remark}% 
\newtheorem{lemma}{Lemma}% 
\newtheorem{assumption}{Assumption}% 
\newtheorem{theorem}{Theorem}% 

\newenvironment{proof}{{\it Proof.}  - }{\hfill \flushright  \qed\bigskip\par}

\begin{document}
\begin{frontmatter}

\title{A fuzzy derivative model approach to time-series prediction\thanksref{footnoteinfo}} 

\thanks[footnoteinfo]{Work  financed by  FCT - Funda\c c\~ao para a Ci\^{e}ncia e a Tecnologia under project:  (i)  UIDB/04033/2020 for the first author.   (ii)  UIDB/00048/2020 for the second author.   }

\thanks[footnoteinfo]{Work  financed by  FCT - Funda\c c\~ao para a Ci\^{e}ncia e a Tecnologia under project:  (i)  UIDB/04033/2020 for the first author.   (ii)  UIDB/00048/2020 for the second author.   }
\author[First]{Paulo A. Salgado} 
\author[Second]{T-P Azevedo Perdicoúlis} 
\address[First]{Department of Engineering, ECT \& CITAB, 
  UTAD, Quinta de Prados, 5000-801 Vila Real, Portugal (e-mail: psal@utad.pt) ORCID 0000-0003-0041-0256.}
\address[Second]{Department of Engineering, ECT \& ISR, University of Coimbra, 
 UTAD, Quinta de Prados, 5000-801 Vila Real, Portugal (e-mail: tazevedo@utad.pt) ORCID 0000-0002-3281-5357.}

\begin{abstract}                % Abstract of not more than 250 words.
This paper presents a fuzzy system  approach to the prediction of nonlinear time-series and dynamical systems. To do this, the underlying mechanism governing a time-series is perceived by a modified structure of a fuzzy system in order to capture the time-series behaviour, as well as the information about its successive time derivatives. The prediction task is carried out by a fuzzy predictor based on the extracted rules and on a Taylor ODE solver. The approach has been applied to a benchmark problem:   the Mackey- Glass chaotic time-series.  Furthermore, comparative studies with other fuzzy and neural network predictors were made and these suggest   equal or even better performance of the herein presented approach.\end{abstract}


\end{frontmatter}
%===============================================================================

\section{Introduction}
%This work relates to  modelling of time-series data, which consists of simultaneous observations of several related variables of interest. The usual objective is  to explore the dynamic relationship  between the variables, and perhaps forecast the system by using a suitable fuzzy system (FS) representation.  
%
A time-series is a discrete sequence of measured quantities $x_{1}, x_{2}, \ldots, x_{n},$  of  some physical system (taken at regular intervals of time) or from human activity data \cite{ref1}. Here, the time-series prediction problem is formulated as a system identification problem, where the system input   is its past values and its desired output   its future values.
Much effort has been devoted over the past several decades to the development and improvement of time-series forecasting models. %A set of analytical and programming tools have been implemented, mostly by looking at trends, moving averages, and certain graphical patterns  for performing predictions and subsequently making decisions. Most of these linear approaches, such as the well-known Box-Jenkins method, have shortcomings \cite{ref2,ref3}, since in  these cases, a linear correlation structure is assumed among the time-series values and, therefore, no nonlinear patterns.  As a result,   when  used on complex real-world problems, the result is not always satisfactory. 
Recently,  has been an increasing interest in extending the classical framework of Box-Jenkins to incorporate nonstandard properties, such as nonlinearity, non-Gaussianity, and heterogeneity \cite{ref4,ref5,ref6}. So far, these solutions present a congenital shortcoming due to their lack of capability to incorporate directly the natural linguistic information in their modelling or in their strategies, or even to extract relevant linguistic information from the data series. 

Moreover, neural networks and fuzzy logic modelling have been applied to the problem of forecasting complex time-series \cite{ref7,ref8,ref9,ref10} with advantages over the traditional statistical approaches \cite{ref11,ref12}, with their major advantage being their flexible nonlinear modelling capability through linguistic rules or weighted neurone network.% In  both cases, there is no need to specify a particular model form, which is clearly needed in the classical regression analysis \cite{ref13}. Rather, the model is adaptively formed and tuned based on the features presented on data. This data-driven approach is suitable for many empirical data sets where no theoretical guidance is available to suggest an appropriate data generating process. Of course, there are also disadvantages when compared to statistical regressions models, where we can use the information given by their parameters to understand the process. This problem can be in great part solved if, in alternative, the relationships of the process are expressed by linguistic relationships, which are transparent and easily readable by an expert. One of the most popular approaches to modelling these processes is the FS based on Fset theory.

From a computational point of view,  fuzzy systems (FS) are inherently nonlinear and have the capability to approximate nonlinear functions \cite{ref14,ref15}.  The key point of the learning mechanism is to show that the combined use of the fuzzy rule-based system with structural learning results is a powerful and efficient tool for the automatic extraction of a set of meaningful rules, as a collection of momentary pictures (shots) of process behaviours \cite{ref16,ref17,ref18,ref19,ref20}. However, the classical FS is based on the belief that there are static linear or non-linear relationships between historical data and future values of a time-series system, here referred as a zero order system. Unfortunately, in many practical situations, the zero order approximation capability of the  FS  is not sufficient to approximate temporal series. In other situations, it will be useful to know the series of derivatives from the original time-series, since the derivative of FS seldom is an approximate of the derivatives' series.

More recently, other works have appeared that aim at approximating time-series using FS  \cite{ref42,2201.02297,TSAUR2005539,CAGCAGYOLCU20168750,8015463}.    Even treating some applications as in \cite{PEREIRA2015395,Castillo}.

This paper  models the time-series data,    exploring the dynamic relationship  between the variables,   using a suitable FS representation.  In particular, the problem of future prediction uses a Disturbed Fuzzy Modelling (DFM) approach, which is a generalised FS capable of approximating regular functions as well as their derivatives on compact domains with linguistic information. Its rules are extracted from available past data or local Taylor series (TS) expansion using a new least-square multivariate rational approximation. This linguistic information is related to the translation process of Fuzzy sets (Fsets) within the fuzzy relationships, which, when modelled, are capable of describing the local trend of the fuzzy models (time or space derivatives). With derivatives' models of time-series, for regions of interest, a TS is able to approximate either a solution of ordinary differential equations (ODE) or time-series in distinct regions of the space. The terms of this TS are now a set of FS  which are derivative functions of the DFM. Such representation is designated as the TS of Fuzzy Functions (TSFF), which are the time-series approximates.

This work is organised as follows: Section~\ref{II} describes the disturbed fuzzy system (DFS); Section~\ref{III}  explains the learning methods to approximate a function and its derivatives by the  DFS. Section~\ref{IV}  presents an ODE Fuzzy   method based on the DFM. Section~\ref{V}  applies the proposed algorithm to  a benchmark problem  and compares the simulation results with other approaches. Finally, conclusions are drawn and some directions for future work are outline in Section~\ref{VI}.

\section{The  disturbed fuzzy system}\label{II}
FS models provide  a framework for modelling complex nonlinear relations  using a rule-based methodology. To do this, consider a system $y = f (x),$  where $y$ is the output (or consequent) variable and $x = \left(x_{1},x_{2}, \ldots, x_{n}\right)^{T} \in \mathbb{R}^{n} $ is the input vector (or antecedent) variable. Let $U =U_{1} \times \cdots \times U_{n}$ be the domain of the input vector $x  \in \mathbb{R}^{n},$  and $V$ the output space.
A linguistic model, relating variables $x$ and $y,$ can be written as a collection of rules that link terms $A_{j_{i}} \in U_{i},$  $i=\overline{1,n},$   %$j_{i}= 1,\ldots , N_{i},$ 
and $B_{j} \in V, $ $j=\overline{1, N,} $  %j=\left(j_{1}, \ldots, j_{n}\right),$ 
where $ A_{j}(x) = \left(  A_{j_{1}}(x_{1}), \ldots, A_{j_{i}}(x_{i}), \ldots,  A_{j_{n}}(x_{n})\right)$ and $B_{j}(y)$  represent  the $j$-descriptor sets associated to variables $x$  and $y,$  respectively.    In FS modelling, this relationship is represented by a collection $\mathcal{T}$ of fuzzy IF–THEN rules:
\begin{eqnarray} 
R_{ j}&: &  \text{ IF } x_{1}\in A_{j_{1}}  \text{ and } \ldots  x_{i}\in A_{j_{i}}   \ldots  \text{ and }   x_{n}\in A_{j_{n}}  \nonumber \\ & & \text{ THEN } y  \in B_{ j}  \label{eq1} 
 \end{eqnarray} 
where  the rule index  $j$ belongs to     set $$J = \left\{j | j=\left(j_{1}, \ldots, j_{n}\right),   j_{i}=\overline{1, N_{i}}, i=\overline{1, n}, j=\overline{1,N}\right\}.$$ 

The input space  $U $  
 and the output space $V$ are completely partitioned into $N$   fuzzy regions, where  $N$ fuzzy rules of   form \eqref{eq1} can be defined (for $U,$ $N=\cup_{i=1}^{n} N_{i} $). The rule-base can be represented by the fuzzy relation defined on the Cartesian product $A\times B.$  If each input space $U_{ i}, i=\overline{1,n},$  is completely partitioned into $N_{i}$ Fsets, then there always exists at least one active rule.

 Next, some relevant concepts are introduced.
\begin{definition}[Completeness]\label{def1}
The collection of Fsets  \\
$\left\{ A_{ j}  \right\}_{j=\overline{1,N}}$ in $U$ is said to be complete on $ U$ if
$\forall x \in U,$     $\exists j| A_{ j}  ( x ) > 0 .$
\end{definition}

\begin{definition}[Support]\label{def2}
The   support of a Fset  $A$  on  $U$ is $Supp(A) = \left\{ x \in U | A(x) > 0 \right\}.$
\end{definition}

\begin{definition}[Core]\label{def3}
The core of a Fset  $A$  on  $U $ is  $C ( A ) = \{ x \in U | A ( x ) =1 \} .$
\end{definition}

\begin{definition}[Consistency]\label{def4}
The collection of Fsets  \\
$\left\{ A_{ j}  \right\}_{j=\overline{1,N}}$ in $U$ is said to be consistent if $A_{ j} ( x ) = 1$ for some $x \in U$ implies that  $A_{j'} ( x ) = 0,  \forall j'\neq j .$
\end{definition}

\begin{assumption} \label{assump2}\
\begin{itemize} 
\item[(i)] Fsets $\left\{ A_{ j_{i}} \right\}_{i=\overline{1,n}}$ %,j _{i}=1,2,\ldots , N_{i},   i =1,2,\ldots , n$ 
are convex, normal, consistent, and complete in $U_{i};$ 
\item[(ii)] Fsets $A_{ j_{i}}$ are ordered between themselves, i.e., $A_{ j_{ i} }< A_{ j_{ k}},  \forall j_{ i} < j_{ k} ;$ 
\item[(iii)]  $A_{ j_{ i}}$ has a membership function $A_{ j_{ i}} ( x_{ i })$ whose\,\, value is zero outside $\left[ {\bar  x}_{ j i-1} ,  {\bar  x}_{ j i+1 }\right] .$ \\[3pt]
It increases monotonically on $\left[ {\bar  x}_{ j i-1} ,  {\bar  x}_{ j i  }\right] ,$  reaches 1 at  $x_{ i} ={\bar  x}_{ j i  },$ %for  $ j_{i} =1,2,\ldots ,  N_{i}, i =1,2,\ldots , n $ 
decreases monotonically next, and becomes zero at   $x_{ i} ={\bar  x}_{ j i+1 } ;$ 
\item[(iv)] The antecedent of rule $R_{j}$   
is a Fset whose membership function is $A_{j}(x)$ that is obtained by $T-$norm aggregations  $( \star ) $  of $A_{ j_{ i}} ( x_{ i }),$ as explained in \eqref{eq3}. \\[3pt] $A_{ j} ( x_{ j} ) =1$ for $x_{j}={\bar  x}_{ j} =\left( {\bar  x}_{j 1}, \ldots, {\bar  x}_{ j n} \right) \in U$ and zero for regions outside of the multi-dimensional interval 
$\times_{i=1}^{n}\left[ {\bar  x}_{ j i-1} ,  {\bar  x}_{ j i+1 }\right] .$ 
\end{itemize}
\end{assumption}
  Since Fsets $ A _{i,1}, A_{ i,2},\ldots , A_{ i,N} \in U_{i}, i = \overline{1, n},$ use membership functions which are normal, consistent, and complete, thus at least one and at most  $2^{n} $ $ A_{ j} ( x )$ are nonzero for every $j \in J .$
\begin{remark} \label{rem3}
As $\left\{ A_{ j}  \right\}_{j=\overline{1,N}}$ in $U$ are consistent, there is a collection of special points, equal to the corresponding central point of $A_{ j} ( x ) ,$  where only one rule can be fired. We define the collection of these special
points as\\ $S = \{ \bar x _{j} = ( \bar x_{ j 1} , \bar x _{j 2} ,\ldots , \bar x _{j n} ) | j \in J \} .$
\end{remark}

 Given the value for the input variables, $x = x^{*}$, the value of $y$ is calculated as a fuzzy subset $G$ using a fuzzy inference process \cite{ref21}:
\begin{enumerate} 
\item[1.] For every rule $R_{j},$ find its firing level:
\begin{equation}\label{eq3} 
A_{j}(x)=A_{j_{1}}(x_{1})   \star A_{j_{2}}(x_{2})  \star \cdots \star A_{j_{n}}(x_{n}). 
\end{equation}
With the linguistic connective “and” in the antecedent of rule \eqref{eq1} defined as a $T-$norm operation, $\star $,  and $A_{j}$ can be viewed as the Fset ${\times}_{i=1}^{n} A_{j_{i}}$ with membership functions $A_{j_{i}}(x) .$
\item[2.] The fuzzy implication of every rule $R_{ j} : A_{j} \mapsto   B_{ j}$ is a Fset in $U\times V$
which is defined as $ R_{j,A_{j} \mapsto   B_{ j} } ( x , y ) = A _{j} ( x ) \times B_{ j} ( y )$ where 
 ``$\times $'' is an operator rule of fuzzy implication, usually min-max inference or arithmetic inference. For each rule $R_{j},$ calculate the effective output value $G_{j },$ based on sup-star composition
\begin{equation}\label{eq4} 
G_{j} = \sup_{x \in U} \left[A'(x) \star  R_{j,A_{j} \mapsto   B_{ j} } (x,y) \right],  
\end{equation}
where $ \star $  
could be any operator in the class of $T-$norm.  $ A'( x )$ is  generally considered as a singleton set (in the singleton fuzzifier we have $A'(x) =1 $).
\item[3.] Combine the individual outputs of the activated rules to find the overall system output $G.$ It   uses the union of these outputs to get the overall output: \begin{equation}\label{eq5} G= \cup_{j=1}^{N} G_{j}.\end{equation}
For the arithmetic inference process, the output of each  $R_{j},$ $G_{ j}  ( y ) = A_{ j }( x ) \times B_{ j} ( y ) .$
\end{enumerate}
In many situations, e.g. in series prediction and modelling applications, it is desirable to have a crisp value $y^{*}$ for the output of a FS  instead of a fuzzy value $G ( y ) .$  This process is accomplished by a defuzziﬁcation mechanism that performs a mapping from the Fsets in $V$ to crisp points that are also in $V.$ In this paper, a center-average defuzzifier  \cite{ref21} is used and the FS output expression is % of the FS is 
\begin{equation}\label{eq6} 
g(x) = \dfrac{\sum_{j=1}^{N}A_{j}(x) \theta_{j} }{\sum_{j=1}^{N}A_{j}(x)},
\end{equation}
where $\theta_{j}$ is the centroid point in $V$ for which the membership function $B_{ j} ( y )$ achieves its maximum value,  assuming that $B_{ j }( y )$ is a normal Fset, i.e., $B _{j }( \theta_{ j} ) =1 .$

\begin{remark}\label{rem1}
As the  Fset collection $\left\{ A_{ j}  \right\}_{j=\overline{1,N}}$ in \eqref{eq6}  is complete for every $x\in U,$ %$\exists j$ such that $A_{ j} ( x ) > 0 ,$ with the FS well defined, i.e.,
 the  denominator is always nonzero.
\end{remark}
\subsection{Design of FS}

\begin{description} 
\item[Step 1] For each input$-i,$ define   Fsets $ A_{ i,j} \in U_{i} ,$ $i = \overline{1, n },$ $j = \overline{1, N },   $  using membership functions
 which are normal, consistent, and complete.   % (for example, triangular membership functions). 
   From the combinatorial
% {\tt min} (or {\tt product}) 
aggregation of these Fsets, results $N=\Pi_{i=1}^{n} N_{i}$ multidimensional Fsets 
 $A_{j}(x) = \left\{ A_{ j1}(x _{1}), A_{ j2}(x_{ 2}),\ldots  , A_{ jn}(x_{ n})  \right\},$ with a central point ${\bar x}_{j}.$ 
\item[Step 2] Construct $N$ fuzzy IF-THEN rules of  form \eqref{eq1}, where   Fset $B_{ j},   j \in J ,$  and $ \theta_{j}$ is chosen as  
%\begin{equation}\label{eq8} 
$\theta_{j} =f ({\bar x}_{j}). $
%\end{equation}
\item[Step 3]  Construct the FS $g (x)$ from the $N$ rules using product inference engine, the  singleton fuzzifier, and the centre average defuzzifier of the form \eqref{eq6}.
\end{description}

\subsection{Disturbed FS}

Fuzzy identification systems are able to integrate information from different sources, namely from human experts and   experimental observation, expressing  knowledge by linguistic IF $\ldots$ THEN rules. However, this translation process of the knowledge into the linguistic IF-THEN rules is made as a static or instantaneous picture of the modelled process, where the dynamical information is discarded. The result is a FS  able to approximate the process transfer function but not of modelling the derivatives' information. %In most   research, the solution found consists of taking an exhaustive fuzzy partition (partition granularity) of the work process space, as well as the establishment of a large number of relationship rules \cite{ref22}  and appropriate membership functions \cite{ref23,ref24}. A greater number of Fsets in a fuzzy partition has as result  a finer granularity. The representation is then more precise but less convenient for linguistic interpretation, but still with few improvements in the field of derivative approximation.

The state variables of a dynamic process or of a time-series are not static, because in each instant they possess an instantaneous value and a trend of evolution. The evolution trend,  whose information is contained in the derivatives, must also be modelled by the FS. A simple way to accomplish this is to add to each Fset an input and output movement trend, here named as  disturbed trends. The main difference between the traditional Fset and the  disturbed Fset lies mainly in the fact that the first is only characterised by its static position and shape while the second has the potential to contain  in its structure also the  velocity, acceleration, etc., of its trends.
In this work, the  DFS reflects a natural trend of FS  for modelling time-series. So, the time-series higher order trends can be modelled by increasing the liberty degrees of the Fsets by a set of transformation operations, for example as result of either its translation in space or a deformation shape, or even both. In the context of this paper, we are concerned with a special type of disturbed Fset: the translation and the additive disturbed transformation.

\begin{definition}[nonlinear translation]\label{def5}
The nonlinear translation of a Fset $A$ on $U$ by $h \in U ,$ denoted $A_{ h},$ is the   Fsubset of $U$ defined   as  $A_{h} (x) = A( x -\sigma_{x} (h) ),$   where 
$\sigma_{x} (h)$ is a nonlinear homogeneous translation function of the disturbance variable $h,$ i. e., $\lim_{h \to 0} \sigma_{x}(h) =0.$  
\end{definition}
Moreover, its values are limited, $  \sigma_{x} (h) \le \dfrac{1-A(x)}{A(x)},$ in order to preserve normality of $A_{h}(x), x \in C(A).$ 
For convenience of representation sake, we consider  $ \sigma_{x} (h) := \sigma  (x,h) $ and $A_{h}(x) := A(x,h). $
Disturbance $h $ moves   Fset $ A $ from its natural position to another position in the neighbourhood. As a special and well-known case we have  that $ \sigma_{x}  (h) = h .$ %So, $\sigma_{x}  (h)$ is a nonlinear function of variables $x$ and $h,$ homogeneous  in   $h,$ i.e.,  $\lim_{h \to 0} \sigma_{x} (h)= 0 . $ 

\begin{definition}[additive disturbance]\label{def6}
Let $\rho (x,h)$ be an additive disturbance function such that $ \rho (x,h) =  \sigma_{x}  (h) A(x).$    The additive disturbed Fset of $A$ is $A_{h}(x) = A(x) \varphi (x,h),$ where $ \varphi (x,h) = 1 +  \sigma_{x} (h). $
\end{definition}

 Both previously defined disturbed Fsets obey the following lemma.
\begin{lemma}\label{lem1} 
$A_{h}(x), h \in \mathbb{R}^{n},$ is the disturbed membership function of Definition~\ref{def5}  and \ref{def6}, that is continuous with respect to  $ h,$ i.e., $\lim_{h \to 0} \left\|  A_{h}(x) -A(x)\right\| = 0.$
\end{lemma}

\begin{proof} 
From Definition~\ref{def5}  and \ref{def6}, is follows directly that $ A_{h}(x)$ is continuous  in $h$ on the $Supp(A)$.
\end{proof}
Remember that $A(x)$ is a combination of the component  membership functions   based on $T-$norm $\star.$  Then, its disturbed couterpart  functions are $A_{ h,j}(x) = A_{ h,j_{1}}(x_{ 1}) \star \cdots \star A_{ h,j_{n}}(x_{ n}) .$  
Moreover, for the arithmetic product $T-$norm operation and disturbances of additive type,  we have:  $A_{h,j}(x)= A_{j}(x) \varphi_{j}(x,h),$ where $ \varphi_{j}(x,h) =  \varphi_{h_{1},j_{1}}(x_{1},h_{1}) \ldots \varphi_{h_{n},j_{n}}(x_{n},h_{n}) $ and $h = \left( h_{1}, \ldots, h_{n}\right)^{T}$ is the disturbance vector.  
Consequently, the fuzzy relationships that involve   Fset A are also disturbed, and this  reflects on  FS. The result % of disturbed Fsets
 is also a DFS that is equal to the static FS when the disturbance variables $h$ are null.
%A useful type of perturbed FS results from the use of distinct functions of additive perturbed type function and nonlinear translation, respectively for input and output Fsets, with the same perturbation vector $h .$

\begin{definition}\label{def7}
A DFS results from the disturbance of the input and   output Fsets of   \eqref{eq6}. Let the input Fsets of rules be of additive type, i.e, $A_{h,j}(x) = A_{ j}(x)\varphi_{j} (x, h),$ and the output Fset of rules of nonlinear translation type, i.e., $B_{ h,j}(y)= B_{j}(y+\sigma_{xj} (h)) .$ Hence, the DFS \eqref{eq6}  is:
\begin{equation}\label{eq7} 
g(x,h) = \dfrac{\sum_{j=1}^{N}A_{h,j}(x) \left( \theta_{j}+\sigma_{xj}(h) \right)}{\sum_{j=1}^{N}A_{h,j}(x)}.  
\end{equation}
\end{definition}
The next objective is to prove that whenever  $f\in {\cal C}^{\nu}(U)$ with  $ U \in  \mathbb{R}^{n}$  a compact set   that is completely partitioned into $N$ Fsets, then for an arbitrary   $ \varepsilon > 0$ there exists a DFS $g (x , h)$  \eqref{eq7}  that approximates $f (x)$ up to the $\nu$th order derivative if 
$\displaystyle \sum_{x \in Supp(f) } \left\| \dfrac{\partial^{i} f(x)}{\partial x^{i}  }  -  \displaystyle \lim_{h \to 0} \dfrac{\partial^{i} g(x,h)}{\partial h^{i}  }   \right\| < \varepsilon ,$ $i=\overline{0,\nu}.$  
That is, the DFS of type \eqref{eq7} is the $\nu$th order   approximate. 

\begin{assumption} \label{assump1}
The  FS on $U$ is given by \eqref{eq6} and its disturbed counterpart by \eqref{eq7}.
\end{assumption} 

The next step is to propose  a method to design a FS    that claims this  property %. To do this, one first needs to design the static FS in a step-by-step manner 
and then study  the accuracy of the approximate.

Further on,  adopting the previous notation, concepts/ definitions,   the following assumptions are in place.

\section{Sufficient condition for a DFS as a  derivatives' approximate}\label{III}
The purpose of  this section is to  prove a sufficient condition for  DFS as a universal approximate of a real continuous and differentiable function up to the $\nu$th order derivative.  Before starting, some notation   and definitions are in place.   %To save space,  we use multi-indices and vector notations. 

 $\mathbb{R}^{d }$ is the $d-$dimensional Euclidean space.   Vectors are represented in bold  font. $\mathbb{Z}_{+}^{n}$ is the set of all non-negative multi-integers. 

 $\boldsymbol \alpha = \left(   \alpha_{1}, \ldots , \alpha_{d} \right) $ is a multi-index where $\alpha_{j}$ are nonnegative integers. Also $ |\boldsymbol \alpha | =  \alpha_{1}+ \cdots + \alpha_{d}$ and $\boldsymbol{  \alpha } !=  \alpha_{1}! \ldots  \alpha_{d}!  $    
 
 Let $\alpha $ and $\beta $ be two multi-indices. If $ \beta_{k} \ge \alpha_{k}, \forall k=\overline{1,d},$ then $\boldsymbol{  \beta} \ge \boldsymbol{ \alpha }.$   $\boldsymbol x \in \mathbb{R}^{n}$ and $\boldsymbol x = \sum_{i=1}^{n} |x_{i}|.$  Then, 
for $ d = ( d_{ 1} , d_{ 2} ,\ldots , d_{ n} )^{T}  \in \mathbb{Z}_{+}^{n} ,$  let $ \boldsymbol{x}^{d}   = x _{1}^{ d_{ 1}} \ldots  x_{ n}^{ d_{ n}} ,$ and for a smooth function $f $ on  $\mathbb{R}^{n},$ let  
  $f^{(d)}(x) = \dfrac{\partial^{d} f(x)}{\partial x^{d}} = \dfrac{\partial^{d} f (x) }{  \partial x _{1}^{ d_{ 1}} \ldots \partial x_{ n}^{ d_{ n}}  },$ with $ \sum_{i=1}^{n} d_{i} =d, $ be the  partial $d$th order derivative of $f.$ 
  Given an open set $ U \subset \mathbb{R}^{n} ,$ let $f \in \cal{C}^d ( U )$   be the set   of functions with its first $\boldsymbol k$   partial derivatives  continuous  in  $U,\boldsymbol  k\in \mathbb{Z}_{+}^{n} ,  \boldsymbol{ k} \le   \boldsymbol{d }.$   The $\delta -$neighbourhood  centered in  $\boldsymbol { \tilde x } \in U $ is defined as  
  \begin{equation}\label{neighbourhood} 
  N_{\delta} = \left\{ \boldsymbol x =     \boldsymbol  {\tilde x} +
   { \boldsymbol  h} : \left\|   \boldsymbol  h  \right\|<  \delta  ,   \boldsymbol x, \boldsymbol { \tilde x }  \in U 
  \right\} .
  \end{equation}
The multivariate polynomial $T$ of degree $  \boldsymbol  t$ defined on a compact set $ U$ can be expressed as: 
\begin{eqnarray}
T_{  \boldsymbol  t}( \boldsymbol x) &= &\sum_{d_{1}=0}^{t_{1}} \sum_{d_{2}=0}^{t_{2}} \cdots \sum_{d_{n}=0}^{t_{n}} z_{d_{1},d_{2},\ldots , d_{n}} h_{1}^{d_{1}} h_{2}^{d_{2}} \ldots h_{n}^{d_{n}}\label{eq9}  \\
&= & \sum_{|d | \ge 0}^{t}z_{ \boldsymbol d}(\boldsymbol x) {  \boldsymbol h}^{ \boldsymbol d}, \nonumber
\end{eqnarray}
with $\sum_{i=1}^{n} t_{i}=    t,   \boldsymbol x= \left(x_{1}, x_{2}, \ldots, x_{n}\right)^{T}.$ 
 Also $$ z_{d_{1},d_{2},\ldots , d_{n}}   ( \boldsymbol x)  = \dfrac{1}{d_{1}! d_{2}! \ldots d_{n}!} \dfrac{ \partial^{d} f (\boldsymbol x)   }{\partial x_{1}^{d_{1}} \ldots \partial x_{n}^{d_{n}}}= \dfrac{1}{d!} \dfrac{ \partial^{d} f (\boldsymbol x)   }{\partial x^{d} }.$$

%The $i$th derivative of the product function $f = N {\cdot} D$ and rational function $f = N / D ,$ with respect to the independent univariable (here omitted for simplicity sake), can be calculated as $ f^{( i ) }  =  \sum_{j=0 }^{i} \binom{i}{j} D^{(i-j)} N^{j} $ and $ f ( i ) =\dfrac{N^{(i)}}{D} - \sum_{j=0}^{i-1}  \binom{i}{j} \dfrac{D^{(i-j)}}{D} f^{(j)},$    respectively.

The design of static FS $g ( x ) $ by choosing the appropriate partition of the input space, the shape of the membership function and its position in the input space $U,$ as well as in the output space $V,$ is relevant to  approximate   function $f ( x ) .$ The derivative information could be included in the fuzzy modelling by associating the potential disturbance of its membership function. Without loss of generality, we   consider   the additive disturbance  function $ \varphi_{j} ( x, h ) = \varphi_{j}  ( h )$ to be independent of variable $ \boldsymbol x. $ Furthermore, it is assumed that the disturbance functions are approximated by multivariate polynomials of the multivariate variable disturbance  $ \boldsymbol  h.$

\begin{definition}\label{def8}
Let the disturbance functions $\sigma_{j}( h )$ and $\varphi_{j}( x, h )$ be as in  Definition~\ref{def5} and \ref{def6}, respectively, in a form of multivariate polynomials of degree $ \boldsymbol r$ and  $ \boldsymbol s$  (with $ \boldsymbol s\le  \boldsymbol r$), respectively, defined on a compact set $U  \subset \mathbb{R}^{n} ,$  i.e.:
\begin{eqnarray}
\hspace*{-0.2cm}Q_{s,j} ( \boldsymbol  h) = \varphi_{ j}   ( \boldsymbol  h)= \sum_{d_{1}=0}^{s_{1}}  \cdots \sum_{d_{n}=0}^{s_{n}} a^{j}_{d_{1}, \ldots , d_{n}} h_{1}^{d_{1}}   \ldots h_{n}^{d_{n}}, \label{eq10}  \\
\hspace*{-0.2cm}P_{r,j} ( \boldsymbol  h) = \theta_{ j} + \sigma  ( \boldsymbol  h)= \sum_{d_{1}=0}^{r_{1}} \! \!  \cdots \! \!  \sum_{d_{n}=0}^{r_{n}} b^{j}_{d_{1}, \ldots , d_{n}} h_{1}^{d_{1}}   \ldots h_{n}^{d_{n}} ,\label{eq11} 
\end{eqnarray}
where $\sum_{i=1}^{n} r_{i} =   r, \sum_{i=1}^{n} s_{i} =    s,$  $a^{j}_{0,\ldots,0}=1,$ $b^{j}_{0,\ldots,0}=\theta_{j}, j=1,\ldots, N.$
DFS \eqref{eq7}  becomes now  a rational function  of variable h:

\begin{equation}\label{eq12} 
g( {\boldsymbol  x}, {\boldsymbol  h}) =\dfrac{\sum_{ j=1}^{N} A_{j} (\boldsymbol  x) Q_{   s,j}(\boldsymbol  h)P_{   r,j}(\boldsymbol  h) }{\sum_{ j=1}^{N} A_{j} (\boldsymbol  x) Q_{s,j}(\boldsymbol  h)}.
\end{equation}
\end{definition}

\begin{remark}\label{rem4}
The numerator of \eqref{eq12} is a weighted sum of polynomials of maximum order $   r +    s$ while the denominator is a weighted sum of polynomials of maximum order $s.$ The total number of parameters of polynomials is $N \times ( r + s ) .$ 
\end{remark}

\begin{remark}\label{rem5}
The $i $th partial derivative of the disturbed functions of Definition~\ref{def8} when $h \to 0$ can be calculated iteratively using rules :
 
\begin{equation}\label{eq13} 
\dfrac{1}{i!} g^{(i) } ( {\boldsymbol  x}) = \dfrac{ \sum_{j=1}^{N} A_{j}  ( {\boldsymbol  x})    c_{i}^{j}}{ \sum_{j=1}^{N} A_{j}  ( {\boldsymbol  x}) }  - \sum_{|k|=0}^{|i|-1}  \dfrac{ \sum_{j=1}^{N} A_{j}  ( {\boldsymbol  x})    a_{ {\boldsymbol i}-{ \boldsymbol k}}^{j}} { \sum_{j=1}^{N} A_{j}  ( {\boldsymbol  x}) } \dfrac{1}{{\boldsymbol k}!} g^{(k)}  ( {\boldsymbol  x})
\end{equation}
where $ g^{(i)}  ( {\boldsymbol  x}) = \displaystyle \lim_{h \to 0} g^{(i)} ( {\boldsymbol  x}, {\boldsymbol  h} ) ,$   $c_{i}^{j} = \sum_{ {\boldsymbol  k} \le  {\boldsymbol  i}}  
a_{ {\boldsymbol i}-{ \boldsymbol k}}^{j} 
b_{{\boldsymbol k}}^{j}$.   %$a_{ { \boldsymbol k}}^{j}}, b_{ { \boldsymbol k}}^{j}}$  are the $k-$coefficient of the $j$th polynomial of \eqref{eq10} and  \eqref{eq11}, respectively. Alternatively, we have:
\end{remark}

From Assumption~\ref{assump2} and Remark~\ref{rem4}, for the points $  { \bar{ \boldsymbol  x} }_{j} $ we have that $g^{(i)}  ( \bar{\boldsymbol  x}_{j}) = b_{i}^{j}, j=\overline{1,  N}, i=\overline{1,  r}. $

We   investigate a sufficient condition for DFS to approximate   $f ( x ) ,$ defined on a compact domain $U,  $   up to a given tolerance $\varepsilon  > 0.$ To do this, first,  we establish  a sufficient condition for the DFS to approximate any real polynomial defined on $U$ (Theorem~\ref{th1}). Then,    combining all the sufficient conditions with the Weiertrass Approximation Theorem, we also obtain sufficient conditions for the DFSs to approximate    $f ( x )$ and its successive derivatives (Theorem~\ref{th2}).

\begin{theorem}\label{th1} 
Let   DFS $g( {\boldsymbol  x}, {\boldsymbol  h}) $ be as in equation~\eqref{eq12}.  It can approximate exactly  any $N$ distinct polynomials of order $r \le \nu  ,$ $T_{r, {\bar{\boldsymbol  x}}_{j}} ( {\boldsymbol  h}) = \sum_{| {\boldsymbol  d} | =1}^{ r} z_{\boldsymbol  d} \left(  \bar{\boldsymbol  x}_{j} \right) {\boldsymbol  h}^{\boldsymbol d},$ in $N$ distinct nodes $ {\bar{\boldsymbol  x}}_{j}  \in S,$ and also  the $i$th derivative of $g( \bar{\boldsymbol  x}_{j}, {\boldsymbol  h}) ,$  with respect to     $h,$    can approximate the $i $th derivative of   $T_{r, {\bar{\boldsymbol  x}}_{j}} ( {\boldsymbol  h}), i,r=\overline{0,\nu},$ i.e., 
\begin{enumerate} 
\item[(i)] Define $E_{0}^{(i)} \left(  \bar{\boldsymbol  x}_{j}\right) := \lim_{ h \to 0 } E_{h}^{(i)}  \left(  \bar{\boldsymbol  x}_{j}\right)  =0. $ 
 \item[(ii)] 
 \begin{eqnarray*}
\left| E_{0}^{(i)} \left(   {\boldsymbol  x} \right) \right| &=&  \left|  
 \dfrac{\partial  ^{i} T_{r, {\boldsymbol  x} } ( {\boldsymbol  h})}{\partial {\boldsymbol  h}^{i} }   
    -  \displaystyle \lim_{h \to 0}   \dfrac{\partial^{i} g( {\boldsymbol  x}, {\boldsymbol  h})  }{\partial {\boldsymbol  h}^{i} }   \right|  \\
    &&  \le \left|  
 \dfrac{\partial  ^{i+1} T_{r,  {\boldsymbol  x} } ( {\boldsymbol  h}),}{\partial {\boldsymbol  h}^{i+1} }   
     \right|_{\infty}  {\boldsymbol  h}, i, =\overline{0,\nu}. 
      \end{eqnarray*}
 \end{enumerate}
 \end{theorem}

\begin{proof} 
(i)  Let $T_{r, {\boldsymbol  x}} ( {\boldsymbol  h}) = \sum_{| {\boldsymbol  d} | =1}^{\nu} z_{\boldsymbol d} \left(  {\boldsymbol  x}  \right) {\boldsymbol  h}^{\boldsymbol  d},$ the polynomial in variable $ {\boldsymbol  h}= {\boldsymbol  x}- \bar{\boldsymbol  x}_{j},   \bar{\boldsymbol  x}_{j}\in S. $ Then $\partial {\boldsymbol  h}= \partial {\boldsymbol  x}.$
The aim is to find polynomials $Q_{s,j} ( \boldsymbol  h), P_{r,j} ( \boldsymbol  h) $ as in  \eqref{eq12} to guarantee the approximate of $T_{r, {\boldsymbol  x}} ( {\boldsymbol  h}).$

 Let $\bar{\boldsymbol  x}_{j} \in S.$  In the neighbourhood of $\bar{\boldsymbol  x}_{j}$ the error of the derivative' approximate is:
\begin{eqnarray*} 
E_{{\boldsymbol  h}}^{(i)} (\bar{\boldsymbol  x}_{j})  = 
\dfrac{\partial^{i} }{\partial {\boldsymbol  h}^{i}}\left(T_{r, \bar{\boldsymbol  x}_{j}} ( {\boldsymbol  h})- g \left(  \bar{\boldsymbol  x}_{j}, {\boldsymbol  h} \right)   \right)\\
  = 
\dfrac{\partial^{i} }{\partial {\boldsymbol  h}^{i}}\left(T_{r, \bar{\boldsymbol  x}_{j}} ( {\boldsymbol  h})-   \dfrac{\sum_{ j=1}^{N} A_{j} (\bar{\boldsymbol  x }_{j}) Q_{   s,j}(\boldsymbol  h)P_{   r,j}(\boldsymbol  h) }{\sum_{ j=1}^{N} A_{j} (\bar{\boldsymbol  x}_{j}) Q_{s,j}(\boldsymbol  h)}
   \right) .
 \end{eqnarray*}
Considering $\theta_{j} =f\left(  \bar{\boldsymbol  x }_{j} \right)$ and Definition~\ref{def8}, we immediately have $b_{0}^{j}=f\left(  \bar{\boldsymbol  x }_{j} \right),$ so that $\lim_{h \to 0}E_{{\boldsymbol  h}}^{(i)} (\bar{\boldsymbol  x}_{j})=0.$  If $P_{   r,j}(\boldsymbol  h)$ is chosen to coincide with $T_{   r,\boldsymbol  x}(\boldsymbol  h),$  then $\lim_{h \to 0}E_{{\boldsymbol  h}}^{(i)} (\bar{\boldsymbol  x}_{j})=0, i=\overline{0,\nu}.$
Next, consider the DFS approximate of $T_{   r,\boldsymbol  x}(\boldsymbol  h) $  and its successive derivatives when $h\to 0,$ i.e.:
\begin{eqnarray*} 
E_{{\boldsymbol  h}}^{(i)} ( {\boldsymbol  x} )& =& T_{r, {\boldsymbol  x} }^{(i)} ( {\boldsymbol  h})  
- \displaystyle \lim_{h \to 0}\dfrac{\partial^{i} }{\partial {\boldsymbol  h}^{i}}  \dfrac{\sum_{ j=1}^{N} A_{j} (\bar{\boldsymbol  x }_{j}) Q_{   s,j}(\boldsymbol  h)P_{   r,j}(\boldsymbol  h) }{\sum_{ j=1}^{N} A_{j} (\bar{\boldsymbol  x}_{j}) Q_{s,j}(\boldsymbol  h)}, 
 \end{eqnarray*}
 ${\boldsymbol  x }={\boldsymbol  h }+{\boldsymbol  x}_{j}.$
As \eqref{eq13}  is equivalent to 
\begin{eqnarray} 
 \dfrac{ g^{(i) } }{i!}( {\boldsymbol  x}) &=& \dfrac{ \sum_{j=1}^{N} A_{j}  ( {\boldsymbol  x})    b_{i}^{j}}{ \sum_{j=1}^{N} A_{j}  ( {\boldsymbol  x}) } \label{eq14} \\
 && - \sum_{|k|=0}^{|i|-1}  \dfrac{ \sum_{j=1}^{N} A_{j}  ( {\boldsymbol  x})    a_{ {\boldsymbol i}-{ \boldsymbol k}}^{j}} { \sum_{j=1}^{N} A_{j}  ( {\boldsymbol  x}) } \left(
 b_{i}^{j} -  \dfrac{1}{{\boldsymbol k}!} g^{(k)}  ( {\boldsymbol  x})
 \right) \nonumber
\end{eqnarray}
Substituting \eqref{eq14} above, it becomes:
 \begin{eqnarray*}
 E_{{\boldsymbol  0}}^{(i)} ( {\boldsymbol  x} )& =& T_{r, {\boldsymbol  x} }^{(i)} ( {\boldsymbol  h})  
- {\boldsymbol  i}! \left(    \sum_{j=1}^{N} p_{j}   ( {\boldsymbol  x} )  b_{i}^{j} -  \right. \\
&& \left. 
\sum_{|k|=0}^{|i|-1}   \sum_{j=1}^{N} p_{j}   ( {\boldsymbol  x} )  a_{i-k}^{j}    \left(
 b_{i}^{j} -  \dfrac{1}{{\boldsymbol k}!} g^{(k)}  ( {\boldsymbol  x}) \right) \right)
  \end{eqnarray*}
 with $p_{j}   ( {\boldsymbol  x} ) = \dfrac{A_{j}  ( {\boldsymbol  x})  }{    \sum_{j=1}^{N}    A_{j}({\boldsymbol  x} ) }.$
 For the $N$ points $ \bar{\boldsymbol  x}_{j},$ we have that $\lim_{ h \to 0} g^{(i)}   \left( \bar{\boldsymbol  x}_{j},   {\boldsymbol  h}  \right) = b_{i}^{j}. 
 $
If $z_{i}^{j} = b_{i}^{j}$ then $T_{r, \bar{\boldsymbol  x}_{j}}^{(i)} ( {\boldsymbol  0}) = \lim_{h \to 0 } g^{(i)} \left( \bar{\boldsymbol  x}_{j},   {\boldsymbol  h} \right)  $ and the error is zero for the approximation of the $i$th derivative.

(ii)  $\bar{\boldsymbol  x}_{j} \notin S.$ Then $ E_{{\boldsymbol  0}}^{(i)} ( \bar{\boldsymbol  x}_{j} )=0.$  Hence
 \begin{eqnarray*}
T_{r, {\boldsymbol  x} }^{(i)} ( {\boldsymbol  h})  
& =& {\boldsymbol  i}! \left(    \sum_{j=1}^{N} p_{j}   ( {\boldsymbol  x} )  b_{i}^{j} -  \right. \\
&& \left. 
\sum_{|k|=0}^{|i|-1}   \sum_{j=1}^{N} p_{j}   ( {\boldsymbol  x} )  a_{i-k}^{j}    \left(
 b_{i}^{j} -  \dfrac{1}{{\boldsymbol k}!} g^{(k)}  ( {\boldsymbol  x}) \right) \right) + e^{(i)},
  \end{eqnarray*}
where $e^{(i)}$ is the error of the $i$th derivative approximation.
After some simple algebraic manipulation, we have: 
\begin{eqnarray}
 e^{(i)} &=& {\boldsymbol  i}!  \sum_{j=1}^{N} p_{j}   ( {\boldsymbol  x} ) \left(  b_{i}^{j} - f^{(i)} ( {\boldsymbol  x} )\right) \label{eqth1} \\
 &&- {\boldsymbol  i}!  \sum_{|k|=0}^{|i|-1}    \sum_{j=1}^{N} p_{j}   ( {\boldsymbol  x} )  a_{r}^{j}    \left( b_{i-r}^{j} -  \dfrac{1}{{\boldsymbol (i-r)}!} g^{(i-r)}  ( {\boldsymbol  x}) \right) .  \nonumber
\end{eqnarray}
Parameters $a_{r}^{j} , r=\overline{1,s},$ are the coefficients of $Q_{j}$ and are chosen to minimise error $ e^{(i)} $ and under the assumption $a_{0}^{j}=1.$  In the worst case, $a_{r}^{j} =0$ and $$\left|  e^{(i)}  \right| \le  {\boldsymbol  i}! \left|   p_{j}   ( {\boldsymbol  x} ) \left(  b_{i}^{j} - f^{(i)} ( {\boldsymbol  x} ) \right) \right|.$$
From the theory of fuzzy approximation \cite{ref25,ref31}, we finally have:
 $$\left|  e^{(i)}  \right| \le  {\boldsymbol  i}!    \sup_{\left\| \boldsymbol  \xi \right\| < {\boldsymbol  h}}   
  f^{(i+1)} ( {\boldsymbol { \xi  }} ) {\boldsymbol h }.
  $$

\end{proof}

\begin{theorem}\label{th2} 
Suppose   $N $   overlapped and equally distributed Fsets   assigned to each input variable of the DFS  \eqref{eq12}.   Then, for any given real   function $f (x)$ defined in $\cal{C}^{\nu}$ and an approximation error bound $\varepsilon  > 0 ,$ a DFS \eqref{eq12}  with disturbed functions of Lemma~\ref{lem1} exists,  that   guarantees:
\begin{itemize}
\item[(i)]  $g(x) =\displaystyle \lim_{h \to 0} g(x,h).$
\item[(ii)]  $\displaystyle \sup_{x \in Supp(f)}  \displaystyle \lim_{ h \to 0} \left\| f(x) - g(x,h) \right\| < \varepsilon .$
\item[(iii)]   $ \displaystyle \sup_{x \in Supp(f)}  \displaystyle \lim_{ h \to 0} 
\left\|  \dfrac{\partial^{i} f(x)}{\partial x^{i}}  -   \dfrac{\partial^{i} g(x,h)}{\partial x^{i}}  
\right\| 
   < \varepsilon  , i = \overline{1,   \nu} .$
\end{itemize}
 \end{theorem}

\begin{proof} 
(i) Considering Definition~\ref{def1} and \ref{def2}  for additive disturbed functions and  the nonlinear translation of Fsets, as well as Lemma~\ref{lem1}, then $\lim_{h\to 0} A_{h}(x)=A(x)$ and $\lim_{h\to 0} B_{h}(x)=B(x),$ which leads \eqref{eq12} to converge to \eqref{eq6} when $\boldsymbol h \to 0.$ 

(ii) From (i), immediately we can write $$\left\| f(x) -\lim_{h \to 0} g(x,h)\right\| = \left\| f(x) - g(x)\right\|.$$  The question of universal approximation has been addressed by many authors \cite{ref25,ref31}.  It has been demonstrated that given $f \in \cal{C}$  then a FS exists with an input-output relationship $g$ such that $\left\|  f -g \right\|_{w} \le \varepsilon ,$ with $\left\| \cdot \right\|_{w} $ the Lebesgue norm of any order $w$ and $\varepsilon \in \mathbb{R}_{0}^{+}.$ 

(iii) Consider % $ {\boldsymbol  h}$ a small disturbance, 
$f \in \cal{C}(\mathbb{R}),$  $x \in U,$
$N_{\delta}  $ as in \eqref{neighbourhood}.  Then, by Weierstrass Approximation Theorem and\cite{ref32}, a polynomial % $T_{r, {\boldsymbol  x}} ( {\boldsymbol  h}),$ as in \eqref{eq9}, 
  always exists on   $N_{\delta}$ that approximates $f({\boldsymbol  x})$ to an arbitrary accuracy.  Consider such polynomial as the first terms of a TS of $f({\boldsymbol  x})$  in   the  neighbourhood $N_{\delta},$ $T_{\nu, {\boldsymbol  x}} ( {\boldsymbol  h})$  as in \eqref{eq9}.
\begin{eqnarray}
\left\|  f( {\boldsymbol  x}+ {\boldsymbol  h})  - T_{\nu, {\boldsymbol  x}} ( {\boldsymbol  h}) \right\|_{\infty}  &< & \varepsilon_{0}   \label{eq17a}\\
 \left\|  f^{(i)}( {\boldsymbol  x}+ {\boldsymbol  h})  - T^{(i)}_{\nu, {\boldsymbol  x}} ( {\boldsymbol  h}) \right\|_{\infty} & < & \varepsilon_{1} ,\quad  i=\overline{1,\nu}, \label{eq17b}
 \end{eqnarray}
where $ \displaystyle \varepsilon_{0} \le  \sup_{\left\|\boldsymbol  \zeta \right\| < \delta} \dfrac{d f (\boldsymbol \zeta )}{dh} \delta $ and $ \displaystyle \varepsilon_{1} \le  {\boldsymbol i}! \sup_{\left\|\boldsymbol  \zeta \right\| < \delta} \dfrac{d f^{(i+1)} (\boldsymbol \zeta )}{dh^{i+1}} \delta .$
From \eqref{eq17a}--\eqref{eq17b} and 
\begin{eqnarray*} 
\displaystyle \left| \dfrac{\partial^{i} f(\boldsymbol x)}{ \partial \boldsymbol x^{i}} - \lim_{h\to 0}\dfrac{\partial^{i} g(\boldsymbol x, \boldsymbol  h)}{ \partial \boldsymbol h^{i}}  \right|  \le 
\left| \dfrac{\partial^{i} f (\boldsymbol x)}{ \partial \boldsymbol x^{i}} - \dfrac{\partial^{i} T_{\nu, \boldsymbol x}(  \boldsymbol  h)}{ \partial \boldsymbol h^{i}}  \right| \\
 +
\left| 
\dfrac{\partial^{i} T_{\nu, \boldsymbol x}(  \boldsymbol  h)}{ \partial \boldsymbol h^{i}}  - \lim_{h\to 0}\dfrac{\partial^{i} g(\boldsymbol x, \boldsymbol  h)}{ \partial \boldsymbol h^{i}}  \right| 
 \end{eqnarray*}
 we conclude that 
 \begin{equation*} 
\displaystyle \left| \dfrac{\partial^{i} f(\boldsymbol x)}{ \partial \boldsymbol x^{i}} -\displaystyle \lim_{h\to 0}\dfrac{\partial^{i} g(\boldsymbol x, \boldsymbol  h)}{ \partial \boldsymbol h^{i}}  \right|  \le  2 {\boldsymbol i}! \sup_{\left\|\boldsymbol  \zeta \right\| < \delta} \dfrac{d f^{(i+1)} (\boldsymbol \zeta )}{dh^{i+1}} \delta .
 \end{equation*}
\end{proof}
%$T_{r, {\boldsymbol  x}} ( {\boldsymbol  h}) = \sum_{| {\boldsymbol  d} | =1}^{\nu} z_{\boldsymbol d} \left(  {\boldsymbol  x}  \right) {\boldsymbol  h}^{\boldsymbol  d},$ 

\section{Fuzzy  Taylor Series ODE Method}\label{IV}

A continuous autonomous stationary and nonlinear dynamic system can be described by a set of ODEs 
\begin{equation}\label{eq18} 
\dfrac{dx(t)}{dt} = F\left( x(t) \right), 
\end{equation}
where $x(t)$ is the vector of the system states and $ F$   the system vector field.   \cite{ref33} has proven that such a system can be well described in an Euclidean reconstruction state-space by means of a static mapping $F $ which transforms $n$ past values of a sampled observable  $x_{k} = x(kT) $  into the next future samples, i.e.,  $x_{ k+1}= f \left( 
x_{ k}, x_{ k-1},\ldots , x_{ k-n} \right) .$

Considering the solution of the   initial value problem,  \eqref{eq18}  together with $x(t_{0}) = x_{ 0}$, one may expand  a TS around $ t_{0} $ and   obtain  a local solution which is valid within its radius of convergence $R_{0}.$  Once the series is evaluated at $t_{1} < R_{0},$ one obtains an approximation for $x(t_{ 1}).$  The solution may  then be extended to   point $t_{ 2},$ and so forth, so that by a process of “analytical continuation” one obtains a piecewise polynomial solution to \eqref{eq18}.    Whenever the derivatives of $x$ can be easily obtained (analytically or numerically), the TS method offers several advantages over other methods.  Namely,   it provides more information than other methods --- this includes derivatives' information, local radius of convergence, and the location of the poles in the  complex plane ---, another advantage is that both the step-size and the order can be easily changed, so that   optimal values  can always be chosen. Finally, TS integration provides a piecewise continuous solution to the ODE, having no need   to interpolate at intermediate points. The simplest one-step method  of order $\nu $ is based on the TS expansion of solution $x(t).$
  Considering $x^{(i+1)} (t), i=\overline{0,  \nu},$   continuous on $[a,b],$ then the Taylor's formula gives: $$x(t_{k+1}) =x(t_{k}) + \sum_{i=0}^{\nu} f^{(i)} \left( t_{k}, x(t_{k})\right)  \dfrac{h^{i}}{i !} + h \mathcal{O}(h^{\nu  }),$$
where $\mathcal{O}(h^{\nu +1})= h \mathcal{O}(h^{\nu }) = f^{(\nu +1) } \left(  \xi_{k} , x(\xi_{k} )\right) \dfrac{h^{ \nu +1 }}{(\nu +1)!},$ $ t_{k} \le   \xi_{k} \le t_{k+1}  , $ and the derivatives of $f$ are defined recursively as %$f ^{(1)}(t , x) = f_{ t}(t , x) + f_{ x}(t , x) f (t , x)$ and  
$$f ^{(i)}(t , x) = f _{t}^{(i-1)}(t , x) + f_{ x}^{(i-1)}(t , x) f (t , x), i=1,2, \ldots$$ 
This result leads to a family of methods known as the TS methods, whose   fuzzy version is  given in   the   algorithm stated in Subsection~\ref{ODETFS}.

\subsection{Fuzzy Taylor  ODE  Solver}\label{ODETFS}
To obtain an approximate solution of order $\nu $ to   ODE   \eqref{eq18} on $[a, b] ,$ let $h =(b-a)/n$ and generate, thus,  the sequences:
\begin{eqnarray} 
x(t_{k+1})  \approx   x(t_{k }) + \hspace*{5cm} \label{eq19} \\
+ \left( g \left(  t_{k} , x(t_{k})\right)+ \cdots + g^{(\nu -1)} \left(  t_{k} , x(t_{k})\right)  \dfrac{h^{(\nu-1)}}{(\nu -1)!}\right) h, 
\nonumber \end{eqnarray}
where $t_{0}=a, t_{k+1}=t_{k}+h,k=\overline{0, n-1}, $    functions $g^{(i)}, i=\overline{0,\nu-1}$    are the DFS approximate of $f $   and its successive first derivatives.

To solve   ODE  \eqref{eq18} using  \eqref{eq19},  it is necessary to estimate (either analytically or numericallly)   the derivatives' values of the dynamical system at each observable sample point $x_{ k}.$ With these values, a DFS can be created  in order to be used in the prediction problem. The result is a FS (based on linguistic representation structure) that describes the time-series, as well its derivatives.

With a set of $N_{ p}= \nu \times N $   points of the time-series ($N$ and $ \nu$ as in Definition~\ref{def8} and below), where we know the values of
the local TS terms (just to the $\nu$th order continuous derivative), a multivariate rational approximation is used to identify the coefficients of the polynomials.

\begin{definition}\label{def9}
Consider the TS expansion of $f (x_{k}) ,$  $x_{ k} \in U,$
\begin{equation}\label{eq20} 
f (x_{k}+h) = \sum_{|t|=0}^{\infty} c_{t} (x_{k}) h^{t}.
\end{equation}
The DFS approximate $g (x , h) $  is   a fuzzy rational  function of degree $\nu = r + s$  in the numerator and $ s $ in the denominator  as in \eqref{eq12}, and whose power series expansion agrees with a power series to the highest possible value of $ f.$   Hence $g (x , h)$ is said to be a fuzzy  disturbed approximate to   series \eqref{eq20} if
%\begin{equation}\label{eq21} 
$g(x,0)\approx f(x) $
%\end{equation}
 and also 
% \begin{equation}\label{eq22} 
$ \left.  \dfrac{\partial^{i} g(x+h) }{\partial h^{i}  }
  \right|_{h=0} \approx   \left.  \dfrac{\partial^{i} f(x+h) }{ \partial x^{i}  }  \right|_{h=0} , 0 \le i \le r.$
%  \end{equation}
 The errors, ${\cal E}_{i} ,$ of these approximations are as in \eqref{eqth1}.
%\begin{eqnarray}
%e^{(i)} = i!   \sum_{j=1}^{N} p_{j}(x) \left( b_{i} -f^{(i)} (x)\right) \hspace{4cm} \label{eq23}\\
%-   i!   \sum_{|r|>0}^{|i|} p_{j} (x) a_{r}^{j} \left(  b_{i-r}^{j} -\dfrac{1}{(i-r)!} g^{(i-r)} (x) \right), |i|=0, \ldots, \nu . \nonumber 
% \end{eqnarray}
\end{definition}
This problem can be seen as an   optimisation problem.  That is, find parameters   $\boldsymbol a$ and $\boldsymbol b$ that minimise the   sequence of functions:
\begin{equation}\label{}
\min_{\boldsymbol{a,b}}   \sum_{k} \left\|   {\cal E}_{i}  (x_{k})    \right\|_{2}, \quad 
  i=\overline{0, r},
\end{equation}
where\begin{eqnarray*}\label{}
{\cal E}_{i} (x_{k})&:= &  \sum_{j=1}^{N} p_{j}(x_{k}) \left( b_{i}^{j} - f^{(i)} (x_{k}) \right)   \\
&&   -\sum_{|r|>0}^{|i|}
p_{j}(x_{k}) a_{r}^{j} \left( b_{i-r}^{j}  -\dfrac{1}{(i-r)!} g^{(i-r)} (x_{k}) \right).   \nonumber
\end{eqnarray*}
%\begin{eqnarray}\label{}
%\min_{a_{i}} &= & \sum_{k} \left\|     \sum_{j=1}^{N} p_{j}(x_{k}) \left( b_{i}^{j} - f^{(i)} (x_{k}) \right) \right. \\
%&& \left.  -\sum_{|r|>0}^{|i|}
%p_{j}(x_{k}) a_{r}^{j} \left( b_{i-r}^{j}  -\dfrac{1}{(i-r)!} g^{(i-r)} (x_{k}) \right)     \right\|_{2}\nonumber \\
%&&  i=0,\ldots,r. \nonumber
%\end{eqnarray}
The solution to this problem can be implemented using the following algorithm.

\subsection*{DFS Learning Algorithm--DFSLA}\label{}
Let $S$ be the training data set. For each point $x \in S,$ we have the value of the function and its $r$ derivatives, 
$C= \left\{  f^{(i)} (x),  i=\overline{0,  r }\text{ and } x \in S\right\}.$ 

\begin{description} 
\item[Step-1]  Choose $N$ appropriate points  ${\bar{ \boldsymbol x}}_{j}  \in S  .$  These are the centres of the fuzzy input membership functions, i.e.,     $A_{j} ( \bar{  \boldsymbol x}_{j} ) =1.$ 
Parameters 
$ {\boldsymbol b_{i}}^{j}: = f^{(i)}  ( {\bar{ \boldsymbol x}}_{j}   ) ,$  $i= \overline{0, r}$ and $j=\overline{0,N}.$\\
\item[Step-2]   Find    parameters     $\boldsymbol a_{i}$ that approximate  the following relationship in mean square sense:

$ \sum_{j=1}^{N} p_{j} (x)a_{i}^{j} \left( b_{0}^{j} -g(x)\right) =    {\cal E}_{r}  (x).$ \\
\item[Step-3] $i \leftarrow i+1;$ if $i \le r$ go to Step-2 else End.
\end{description}
Step-1 is the solution of the zero order approximation problem of function $f.$ In this case, $a_{0}$ is a vector of unity value. The combination of the DFLSA algorithm with the ODE Taylor Fuzzy solver Algorithm is here designed as ODE-DFS algorithm.

\section{Numerical Example}\label{V}
%Work in time-series FS modelling   concentrated on forecasting future developments of the time-series from known values of $x$ up to the current time. Formally, this can be stated as: 
Find a function  $f : \mathbb{R}^{n} \times  \mathbb{R}^{m}  \to  \mathbb{R}$  to obtain an estimate of $x$ at time $k + h,$ from the $n$ past time steps:
\begin{equation}\label{eq25} 
x ( k + h ) = f _{h} \left( y ( k ) , u ( k ) \right) ,
\end{equation}
where $u(k) \in \mathbb{R}^{m}$  is an independent variable vector, assumed to be known, and     $y(k):=\left(x(k), \ldots, x(k-n) \right)^{T}.$

If   function $f (y) \in {\cal C(\mathbb{R})}$ in \eqref{eq25} can be written as  \eqref{eq19}, where the DFS model is used to represent the nonlinear derivatives' functions of   dynamic systems with   nonlinear autoregressive exogenous input (NARX) structure. The  disturbed fuzzy model identified by the DFSLA algorithm was used in the ODE Taylor fuzzy solver to
iteratively generate the model output $\hat{x}(k).$ Given the same initial condition of the real model, this method was used  to   generate iteratively the model output $\hat{x}(k+1),$   for the input $\hat{y}(k)$ where the past system output terms $x(\cdot)$ were
replaced by model predictions  $ \hat{x} (\cdot).$  This approach has been evaluated for the problem of Mackey-Glass chaotic time-series prediction.
\subsection{Mackey-Glass chaotic time-series}\label{}
The Mackey-Glass time-series has been widely used as a standard benchmark to assess prediction algorithms.  This time-series is generated by integrating the delay differential equation, $x'(t)=f(y(t)),$ where $$f(y(t))=a x (t -\tau) \left( 1+x^{c}(t-\tau)\right)-bx(t)$$ and $y(t)=\left( x(t) - x (t -\tau)  \right)^{T}.$  With $c=10,$ $a=0.2,$ $b=0.1$ and $\tau = 17,$ the time-series is chaotic, exhibiting a cycle but not periodic behaviour.  The upper order time derivatives of the state variable  $x(t)$ can be defined recursively as: $$x''(t)=f'_{x(t)}x'(t)+f'_{x(t-\tau)}x'(t-\tau)$$ and $$x'''(t)=f'_{x(t)}x''(t)+f''_{x(t-\tau)}x'(t-\tau)^{2}+  f'_{x(t-\tau)}x''(t-\tau). $$

The numerical solution of the ODE is obtained by the fourth-order Runge-Kutta method with time step $0.1,$ initial condition $x(0) = 1.2$ and assuming $ x(t) = 0,   t < 0 .$ The generated  time-series
has  1000 data points, 500 of which were used as training patterns and the other 500 as test data. To build the fuzzy model of Mackey–Glass time Series four variables, $x(t -18) , x(t-12) , x(t-6) , x(t),$ were 
selected as   input variables of the fuzzy model. The interval of these input variables, $[0.40,1.32],$ was partitioned with 3 triangular membership functions. From a total of $3^{4}$ possible rules, we select the $61$ fuzzy rules  more fired to describe the fuzzy model. The DFS model had $r =3$ and $s = 3,$ with capabilities of approximating  $ f  $   until the 3rd derivative term, which was used in the ODE-DFS algorithm.

The results of computational simulation are given in Fig.~\ref{figure1}. The top figure   shows the comparison between the output of the ODE-DFS model and the training points, where the blue line represents the training points and the black line represents the series forecasting. The bottom figure shows the diagram error between the output of ODE-DFS model and the training points.

\begin{figure}[ht]
\begin{center}
\includegraphics[width= 0.5\textwidth]{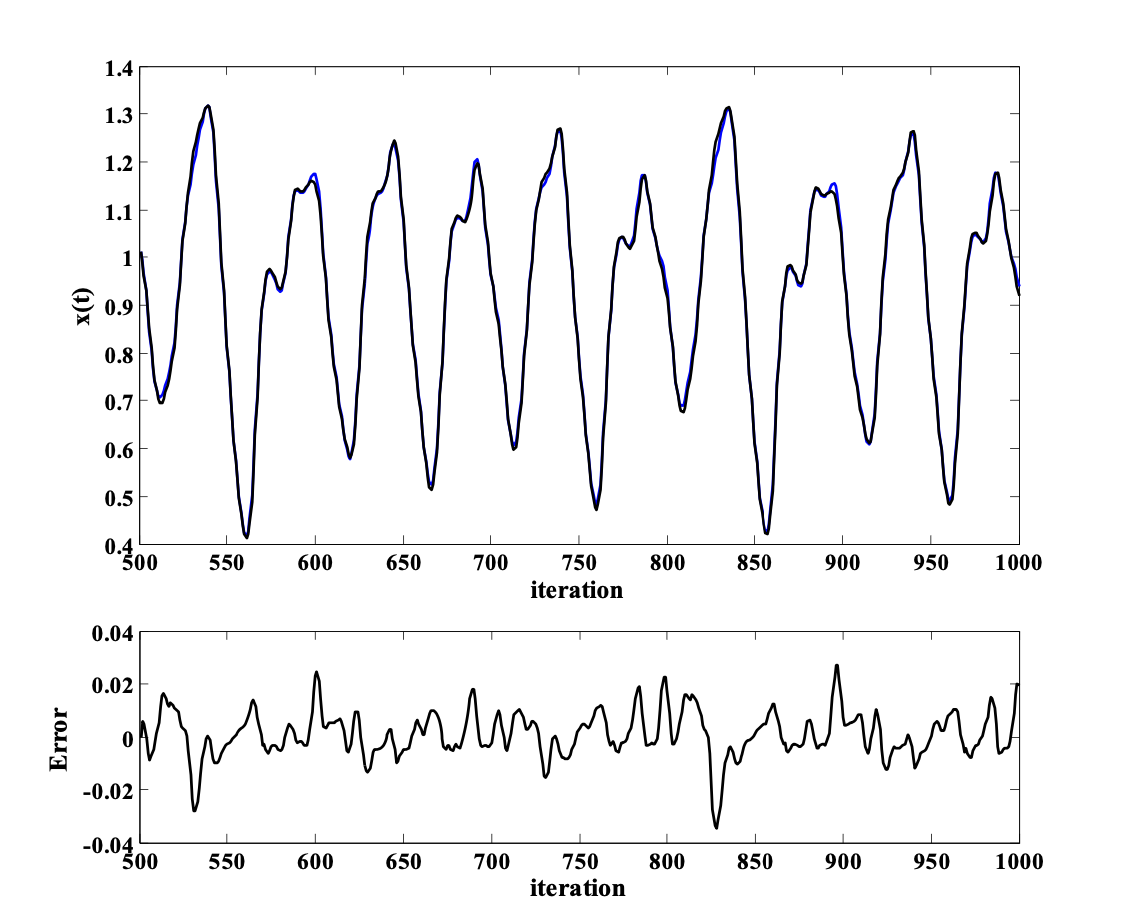}
\end{center}
\caption{Top: Prediction output of   ODE-DFS model (black line) and the testing points (blue line); Bottom: Residual prediction errors.}\label{figure1}
\end{figure}

\begin{table}[ht]
\begin{center}
\includegraphics[width= 0.5\textwidth]{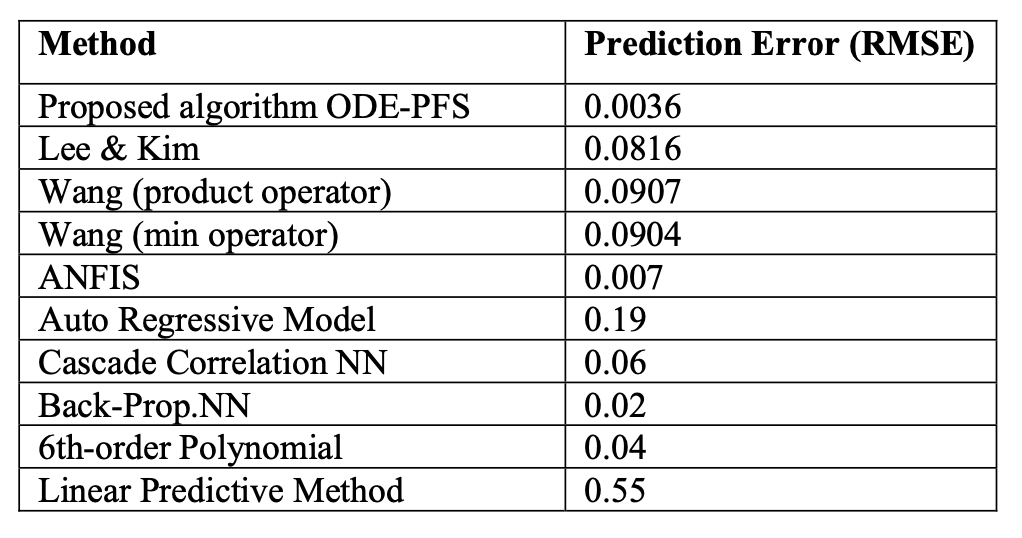}
\end{center}
\caption{Comparison Results of RMSE  among various FPs.}\label{table1}
\end{table}  

\begin{figure}
\begin{center}
\includegraphics[width= 0.5\textwidth]{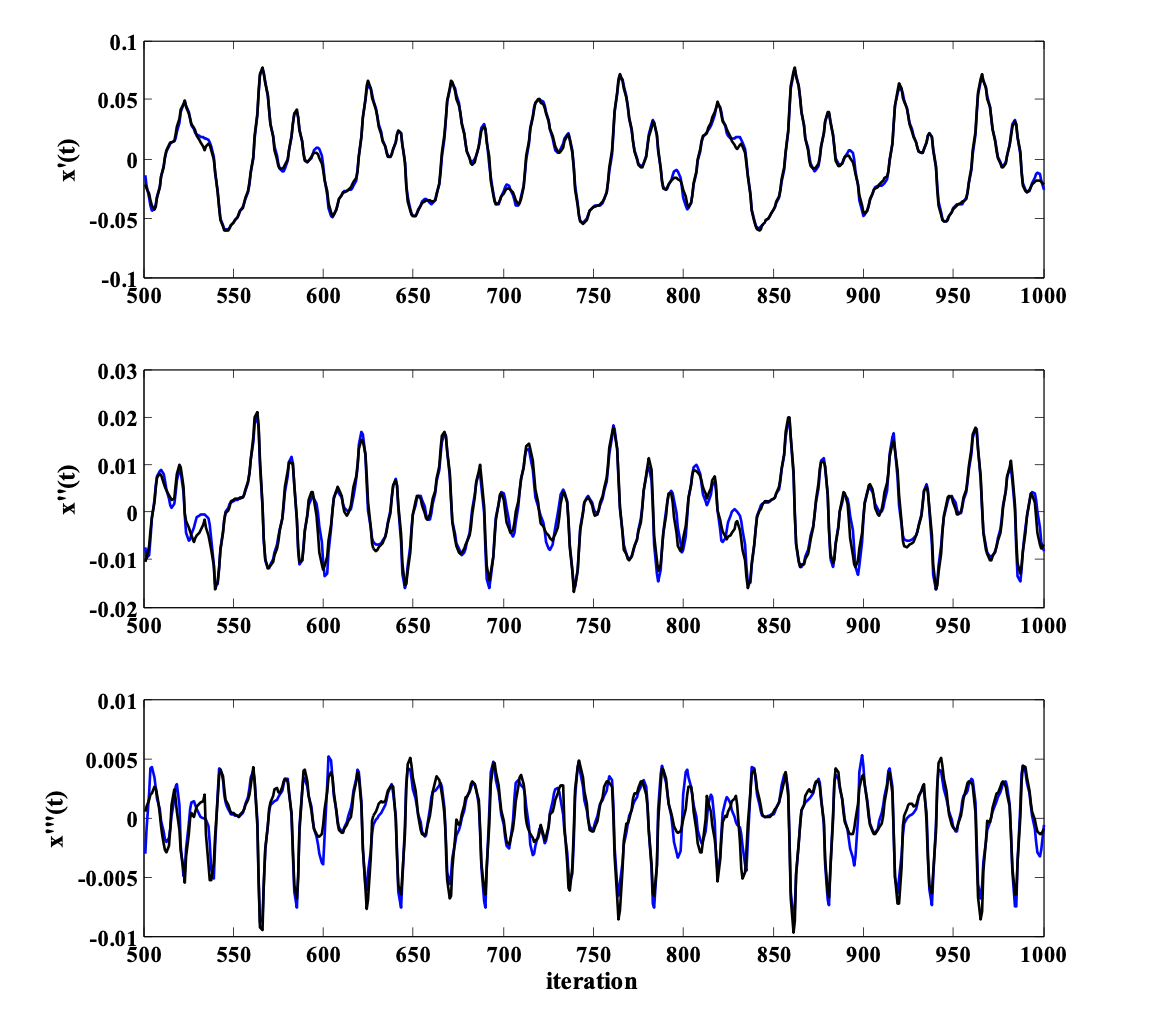}
\end{center}
\caption{Derivatives' time-series $x'(t) , x''(t)$ and $x'''(t)$ (blue line – real series; black line – DFS series forecasting).}\label{figure2}
\end{figure}

Table~\ref{table1} reports the comparison of our fuzzy prediction method with other prediction methods, using the same data \cite{ref35}. For this example, it can be seen that  the performance of the presented method is superior to all the others listed on the table. Even so, our proposed method does not realise the optimisation of the partition of the input  space. This methodology provides also the forecasting of the derivatives' time-series. Fig.~\ref{figure2}  represents the derivative time-series $x'(t) , x''(t)$ and $x'''(t) ,$ whose respective approximation MSE errors are $8.964\times 10^{-7}, 2.353\times 10^{-6}$  and $9.713\times 10^{-7}.$  

\section{Conclusions and future work}\label{VI}
This paper presents a new methodology for prediction of complex discrete time systems or time-series. The approach is based on a  DFS, combined with the ODE Taylor method.
The DFS is a generalisation of the traditional FS that can incorporate relationships of the series in its linguistic fuzzy. In this way, DFS is capable of approximating regular functions, as well as the derivatives up to a given order, on compact domains. 
When the DFS model and its derivatives' models are combined with the ODE Taylor method, the result is  a capable   algorithm to solve forecasting problems.  This methodology was tested with the  Mackey-Glass chaotic time-series forecasting problem.   
Comparative studies were carried out with other fuzzy and neural network predictors that    suggest that our approach can offer comparable or even better performance.
From our study, we conclude that efficiency in accurate and robust forecasting cannot rest solely on a good algorithm. For this reason,  the method proposed in this work has the advantage of capturing and making explicit the derivatives of the time-series, which seems to us to be   a quite desirable feature.

To continue the work, further comparison studies  need to be done using 
 other benchmark examples.  
Also, the performance of Taylor methods of higher order need to be investigated. 

\section{Acknowledgements}\label{}
The authors would like to thank the reviewers for the valuable suggestions.
Work  financed by  FCT - Funda\c c\~ao para a Ci\^{e}ncia e a Tecnologia under project:  (i)  UIDB/04033/2020 for the first author.   (ii)  UIDB/00048/2020 for the second author.

\bibliography{MTNS2022BibFile}
                                                                       % in the appendices.
\end{document}